\documentclass{amsproc}
\usepackage{amsrefs}
\usepackage[all]{xy}
\newtheorem{theorem}{Theorem}[section]

\newtheorem{corollary}[theorem]{Corollary}
\theoremstyle{definition}
\newtheorem{definition}[theorem]{Definition}
\newtheorem{example}[theorem]{Example}
\newtheorem{question}[theorem]{Question}

\usepackage{amssymb}
\theoremstyle{remark}

\newcommand{\Hom}[2]{\mathrm{Hom}({#1},{#2})}
\newcommand{\CC}{\mathbb{C}}
\newcommand{\cF}{\mathcal{F}}
\numberwithin{equation}{section}

\begin{document}

\title{On the semiprime smash product question}

\author{Christian Lomp}
\address{Department of Mathematics, FCUP, University of Porto, Rua Campo Alegre 687, 4169-007 Porto, Portugal}
\email{clomp@fc.up.pt}
\thanks{This is a revision of an already published version of this paper. The author would like to thank Susan Montgomery for pointing out a severe mistake in referring to a survey article by M.Cohen from 1985 and not to the paper by M.Cohen and D.Fishman from 1986 where the central problem of this paper had been originally addressed. He would also like to thank her for her reference regarding the eight dimensional non-trivial semisimple Hopf algebra. Furthermore the author would like to express his sincere thanks to Alveri Sant'Ana for his valuable remarks on a preliminary version as well as to the anonymous referee for his comments which improved this paper. This research was funded by the European Regional Development Fund through the programme COMPETE and by the Portuguese Government through the FCT - Funda\c{c}\~{a}o para a Ci\^encia e a Tecnologia under the project PEst-C/MAT/UI0144/2011.}

\date{}

\dedicatory{dedicated to John Clark and Patrick Smith.}


\begin{abstract}
This is a survey article on a question, posed in 1986 by M.Cohen and D.Fishman, whether  the smash product $A\#H$ of a semisimple Hopf algebra and a semiprime left $H$-module algebra $A$ is itself semiprime.
\end{abstract}

\maketitle


\section{The question}
This paper is a survey on the following question which was raised by M.Cohen and D.Fishman in \cite{CohenFishman} (see also \cite{Cohen85}) and which, to 
the best of my knowledge, has not yet been completely answered.
\begin{question}[Cohen-Fishman, 1986]
Is the smash product $A\# H$ semiprime in case  $H$ is a semisimple Hopf algebra acting on a semiprime algebra $A$ ?
 \end{question}
 
Recall that a ring $R$ is semiprime if its prime radical, the intersection of its prime ideals, is zero. Equivalently 
$R$ is semiprime if and only if it has no nilpotent non-zero ideals.

\subsection{Preliminaries}
After  Bergman's article \cite{Bergman} ``Everybody knows what a Hopf algebra is'', Hopf algebras could be considered  a well-established part of algebra.
Nevertheless I shall briefly review Hopf algebras and their action on algebras. 
Standard texts on this subjects are \cites{Sweedler, Abe, Montgomery, Schneider, Kaplansky, BrzezinskiWisbauer}. 
Let  $k$ be a  field of characteristic zero.\footnote{Many results described in this survey are true or have an analogous description in positive characteristic. However to keep the survey simple I will only report on the characteristic $0$ case.}  Unadorned tensors $\otimes$ are considered over $k$. A Hopf algebra is an 
algebra and a coalgebra satisfying certain relations. Hence let us first define the notion of a coalgebra over a field 
$k$.
\begin{definition}
A coalgebra $(C,\Delta, \epsilon)$ is a $k$-vector space $C$ with $k$-linear maps $\Delta:C\rightarrow C\otimes C$  resp. $\epsilon:C\rightarrow k$, called the comultiplication and counit of $C$ respectively, such that the following diagrams are commutative.
ww\end{definition}
I will denote an (associative, unital) algebra $A$ over $k$ by a triple $(A,\mu,\eta)$ where $\mu:A\otimes A\rightarrow A$ is the 
multiplication and $\eta:k\rightarrow A$ is a $k$-linear algebra map with $\eta(1)=1_A$ being the identity of $A$.
\begin{definition}
Given a coalgebra $(C,\Delta, \epsilon)$ and an algebra $(A,\mu,\eta)$ the space of linear maps $\mathrm{Hom}(C,A)$ becomes an algebra via the  convolution product:
\[\xymatrix{f\ast g  =\mu\circ (f\otimes g)\circ \Delta: &C \ar@{->}[r]^{\Delta} & C\otimes C \ar@{->}[r]^{f\otimes g} &A\otimes A\ar@{->}[r]^\mu & A}\]
for all  $f,g \in \mathrm{Hom}(C,A) $, while the  identity element of $\mathrm{Hom}(C,A)$ is $\eta\circ \epsilon$.
\end{definition}

A Hopf algebra $H$ is an algebra and a coalgebra with certain compatibility conditions and the inverse element of the identity map in the convolution algebra $\mathrm{End}(H)$.

\begin{definition} A Hopf algebra $(H,\mu,\eta,\Delta,\epsilon,S)$ over $k$ is an algebra $(H,\mu,\eta)$ and a coalgebra 
$(H,\Delta, \epsilon)$ such that $\Delta$ and $\epsilon$ are algebra maps and 
$id_H$ has an inverse in $(\mathrm{End}(H),\ast, \eta\circ\epsilon)$ which is denoted by $S$ and called the {\it antipode} of 
$H$.
\end{definition}

\begin{example} Let $(H,\mu,\eta,\Delta,\epsilon,S)$ be a finite dimensional Hopf  algebra and $H^*=\Hom{H}{k}$. Then 
$(H^*,\Delta^*, \epsilon^*, \mu^*,\eta^*, S^*)$ is also a Hopf algebra. Where $f^*:V^*\rightarrow W^*$ is the 
transpose map of a linear map $f:W\rightarrow V$ of finite dimensional spaces defined by $f^*(\varphi)=\varphi\circ f$, 
for all $\varphi\in V^*$. Obviously $k^*$ is identified with $k$. In particular this means that for $\eta(1)=1_H$ one 
has that for all $f\in H^*$, $\eta^*(f)=f\circ \eta \in k^*$ which is identified with $f(1_H)$. Moreover if $\{g_1, \ldots, g_n\}$ is a basis for $H$ and $\{p_1, \ldots p_n\}$ is a dual basis for $H^*$, then the comultiplication $\mu^*$ of $H^*$ takes the form
\begin{equation}\label{dualHopfalgebraComultiplication} \mu^*(f) = \sum_{i=1}^n \sum_{j=1}^n  f(g_ig_j) p_i\otimes p_j \qquad \forall f\in H^*.\end{equation}
\end{example}

For this survey, the two most important examples of finite dimensional Hopf algebras are the following ones. 

\begin{example}
Let $G$ be a finite group. 
\begin{enumerate}
\item The group algebra $k[G]$ is a Hopf algebra with 
$$\Delta(g)=g\otimes g,\qquad \epsilon(g)=1, \qquad S(g)=g^{-1} \qquad \forall g\in G.$$
\item The dual group algebra $k[G]^\ast =\Hom{k[G]}{k}$ with dual basis $\{p_g\}_{g\in G}$ is a Hopf algebra 
where
$$\Delta(p_g)=\sum_{h\in G}p_{gh^{-1}}\otimes p_h,\qquad  \epsilon(p_g)=\delta_{e,g}, \qquad  S(p_g)=p_{g^{-1}} \qquad 
\forall g\in G.$$
Here $e$ denotes the identity element of $G$ and $\delta_{x,y}$ is the Kronecker symbol. Note that $p_g \in k[G]^\ast $ is defined as $p_g(h)=\delta_{g,h}$ for any $g,h\in G$. Since the  multiplication of $k[G]^\ast$ is given by the convolution product, the elements $p_g$ are pairwise orthogonal idempotents. Thus as an algebra $k[G]^\ast$ is simply the direct product of $|G|$ copies of  $k$.
\end{enumerate}
\end{example}

\begin{definition}
A finite dimensional Hopf algebra is called \emph{trivial} if it is isomorphic as a Hopf algebra to a group algebra or a dual group algebra.
\end{definition}

\subsection{Hopf algebra actions}
The category $H$-Mod of left $H$-modules is a tensor (or monoidal) category. More precisely  given two left $H$-modules 
$V$ and $W$ their tensor
product $V\otimes W$ becomes again a left $H$-module by the following action: 
$$h\cdot (v\otimes w) =\Delta(h)(v\otimes w) =\sum_{(h)} (h_1\cdot v)\otimes (h_2\cdot w),$$ 
for all $h\in H, v\in V, w\in W.$ The identity object of the category $H$-Mod is $\mathbb{I}=k$, since the natural isomorphisms $V\otimes k \simeq k\simeq k\otimes V$  are morphisms in $H$-Mod, where $H$ acts on $k$ via the counit  $\epsilon$.
An algebra $(A,\mu, \eta)$ in a tensor category $(\mathcal{C},\otimes, \mathbb{I})$ is an object in this category having morphisms
$\mu:A\otimes A\rightarrow A$ and 
$\eta:\mathbb{I}\rightarrow A$ in $\mathcal{C}$ which satisfy the usual associativity and unity constraints.

\begin{definition}
An algebra $A$ in the category of left $H$-modules is called a (left) $H$-module algebra.
\end{definition}

This means that an ordinary $k$-algebra $(A,\mu,\eta)$ is a left $H$-module algebra if it has a left $H$-module action $\cdot: H\otimes A \rightarrow A$, such that 
$\forall h\in H, a,b\in A$:
\begin{equation}\label{modulealgebra1} 
h\cdot (ab) =h\cdot \mu(a\otimes b) =\mu(h\cdot (a\otimes b))=\sum_{(h)} \mu( (h_1\cdot a)\otimes (h_2\cdot b) )  =
\sum_{(h)} (h_1\cdot a)(h_2\cdot b),
\end{equation}
\begin{equation}\label{modulealgebra2}
h\cdot 1_A =h\cdot \eta(1_k) =\eta(h\cdot 1_k) =\eta(\epsilon(h)) =\epsilon(h) 1_A.
\end{equation}

\begin{example}
Let $G$ be a group and $A$ a $k$-algebra. Then the possible 
left $k[G]$-module algebra structures on $A$ correspond to  group homomorphisms $\psi:G\rightarrow \mathrm{Aut}(A)$ from 
$G$ 
to the group of automorphisms of $A$. To see this, note that if $A$ is a left $k[G]$-module algebra, then equations
(\ref{modulealgebra1}) and (\ref{modulealgebra2}) show that for each $g\in G$ the map $\psi(g):A\rightarrow A$ with 
$\psi(g)(a) =g\cdot a$ for all $a\in A$,
is an automorphism of $A$ with inverse $\psi(g^{-1})$. That $\psi$ is a homomorphism of groups follows from $A$ being a 
$k[G]$-module.
It is easy to check that any such group homomorphism $\psi$ will give raise to a $k[G]$-module algebra structure; a particular case is if $\psi(g)=id_A$ for all $g\in G$.
\end{example}

\begin{example}
Let $G$ be a finite group and $A$ a $k$-algebra. Then the possible 
left $k[G]^*$-module algebra structures on $A$ correspond to possible gradings of $A$ by $G$, i.e. $A=\bigoplus_{g\in G} 
A_g$ with 
$A_gA_h\subseteq A_{gh}$. To see this, note that if $A$ is a left $k[G]^*$-module algebra, then define $A_g :=p_g\cdot 
A$. Since the $p_g$'s 
form a complete set of orthogonal idempotents, $A=\bigoplus_{g\in G} A_g$. Moreover since $p_h\cdot A_g =0$ if $h\neq 
g$, one has that 
$A_gA_h\subseteq A_{gh}$ using equation (\ref{modulealgebra1}). Equation (\ref{modulealgebra2}) implies that $1_A\in 
A_e$.
On the other hand, given any grading $A=\bigoplus_{g\in G} A_g$ by $G$ one defines $p_g\cdot a =\pi_g(a)$, for all 
$g\in G, a\in A$, where $\pi_g$
is the canonical projection of $A$ onto $A_g$.  In particular the trivial grading $A=A_e$ and $A_g=\{0\}$, for all 
$g\neq e$, yields a left $k[G]^*$-module algebra structure.
\end{example} 

As seen,  trivial Hopf algebras act either as group actions or as group gradings. A recent result of Etingof and Walton, \cite{EtingofWalton} shows that over an algebraically closed field of characteristic $0$ any action of a semisimple Hopf algebra $H$ on a (commutative) integral domain $A$ is virtually an action of a group algebra:

\begin{theorem}[Etingof-Walton, 2013]\label{EtingofWaltonTheorem}
Let $A$ be an integral domain and $k$ an algebraically closed field of characteristic $0$. For any action of a semisimple Hopf algebra $H$ on $A$ exists a Hopf ideal $I$ of $H$ and 
a group $G$ such that 
$$ I\cdot A =0 \qquad \mbox{ and } \qquad  H/I \simeq k[G].$$
\end{theorem} 
According to the terminology of \cite{EtingofWalton}, a Hopf algebra $H$ acts inner faithfully on an algebra $A$ if there does not exist a non-zero 
Hopf ideal  $J$ of $H$ with $J\cdot A=0$.
It is clear that given any Hopf algebra $H$ acting on $A$, the sum $I$ of all those Hopf ideals $J$ with $J\cdot A=0$ is again a Hopf ideal 
and hence  $H/I$ acts inner faithfully on $A$.  

As pointed out in \cite[Proposition 5.4] {EtingofWalton}, theorem \ref{EtingofWaltonTheorem} can be extended to 
filtered algebras $A$ such that the Hopf action preserves the filtration and such that the associated graded algebra $\mathrm{gr}(A)$ is a  commutative integral domain.

\subsection{Smash product}
\begin{definition}[Smash product] 
The smash product $A\# H$ of a Hopf algebra $H$ and a left $H$-module algebra $A$ is 
defined on the tensor product $A\# H :=A\otimes H$ with multiplication:
$$(a\# h)(b\#g) =\sum_{(h)} a(h_1\cdot b) \# h_2g,$$
and identity $1_A \# 1_H$.
\end{definition}

It is not difficult to show that this multiplication equips $A\# H$ with  a well-defined, associative and unital 
$k$-algebra structure.
For $H=k[G]$ and a left $k[G]$-module algebra, the smash product $A\# H$ is known as the skew group ring of $A$ and $G$ 
and is denoted by $A* G$.
For the trivial action, $g\cdot a =a$ for all $g\in G,a\in A$, one recovers the group algebra $A*G=A[G]$ of $G$ over $A$.
For group gradings, the smash product $A\# k[G]^*$ has been used for instance in \cite{CohenRowen}.
The algebra $A$ becomes a left $A\# H$-module structure by the  action 
$(a\# h) \bullet b =a(h\cdot b)$ for all $a,b\in A, h\in H$. 
The left $A\# H$-submodules of $A$ are precisely the left  ideals of $A$ that 
are stable under the $H$-action. 

\begin{definition}[Subalgebra of Invariants]
Let $M$ be a left $H$-module. The subspace $M^H =\{m\in M \mid h\cdot m =\epsilon(h)m, \:\:\forall h\in H\}$ is called 
the subspace of 
$H$-invariant elements of $M$. For $M=A$ a left $H$-module algebra, the subspace $A^H$ becomes a subalgebra of $A$, 
called the {\it subalgebra of invariants}.
\end{definition}

For a left $A\# H$-module $M$, there is a natural identification of $M^H$ as the set of left $A\# H$-linear maps from 
$A$ to $M$ via the isomorphism:
$$ \mathrm{Hom}_{H}(A,M) \rightarrow M^H \qquad f\mapsto (1)f, \:\:\: \forall f\in \mathrm{Hom}_{H}(A,M),$$
where homomorphisms of left modules are written opposite of scalars. For $M=A$ one obtains $ \mathrm{End}_{A\# H}(A) 
\simeq A^H$.
In particular $(A\# H, A, (A\# H)^H, A^H)$ is the standard Morita context of the left $A\# H$-module $A$ via the 
identification of $A^H$ with $\mathrm{End}_{A\# H}(A)$ and $(A\# H)^H$ with $\mathrm{Hom}_{A\# H}(A,A\# H)$ (see \cite{CohenFishmanMontgomery}).

The origin of the Cohen-Fishman question stemmed from the next two results, which in characteristic $0$ can be stated as follows:

\begin{theorem}[Fisher-Montgomery, 1978; Lorenz-Passman, 1980]\label{FisherMontgomeryTheorem}
Let $A$ be a $k[G]$-module algebra, with $G$ a finite group. Then $A\# k[G]$ is semiprime if $A$ is $G$-semiprime.
\end{theorem}
Here $G$-semiprime means that $A$ does not contain any non-zero nilpotent $G$-stable ideals.
For the proof see \cite{FisherMontgomery, LorenzPassman}. Actually their proof deals more generally with algebras not necessarily over fields, such that $A$ does not have $|G|$-torsion. The proof given by Lorenz and Passman uses the fact that the skew group ring $A*G$ is a finite normalizing extension of $A$. 
\begin{theorem}[Cohen-Montgomery, 1984]\label{CohenMontgomeryTheorem}
Let $A$ be a $G$-graded $k$-algebra, with $G$ a finite group.  
Then $A\# k[G]^\ast$ is semiprime if $A$ is $G$-graded semiprime.
\end{theorem}
For the proof see \cite{CohenMontgomery}. Here $G$-graded semiprime means that $A$ does not contain any non-zero nilpotent $G$-graded ideals.

Summarizing: theorems \ref{FisherMontgomeryTheorem} and \ref{CohenMontgomeryTheorem} show that Cohen-Fishman's question has a positive answer for  actions of trivial Hopf algebras. 

\subsection{Semisimple Hopf algebra}

Suppose that for a fixed finite dimensional Hopf algebra $H$ the following implication holds for any left $H$-module algebra $A$:
\begin{equation}\label{implication}
A \:\mbox{ semiprime} \qquad \Longrightarrow \qquad A\# H\: \mbox{semiprime.}
\end{equation} 
Considering $A=k$, one concludes that  $H\simeq k\# H$ is semiprime and hence semisimple Artinian. Let me call $H$ \emph{strongly semisimple} if it is finite dimensional and satisfies $(\ref{implication})$ for all left $H$-module algebras $A$.  Hence Cohen-Fishman's question can be reformulated as "Are all semisimple Hopf algebras strongly semisimple ?". Therefore it is important to have a closer look at semisimple Hopf algebras. Larson and Radford characterized in {\cite{LarsonRadford}} semisimple Hopf  algebras over a field of characteristic zero. 
\begin{theorem}[Larson-Radford, 1988]
The following are equivalent for a Hopf algebra $H$ over $k$:
\begin{enumerate}
\item $H$ is a semisimple algebra;
\item $H^*$ is a semisimple algebra;
\item $S^2=id$.
\end{enumerate}
\end{theorem}

\subsection{Blattner-Montgomery duality}

Let $A$ be a left $H$-module algebra over a finite dimensional Hopf algebra  $H$. The smash product $B=A\# H$ has the structure of a left $H^*$-module algebra given by 
$$f\cdot (a\# h) := \sum_{(h)} a \# f(h_2) h_1, \qquad \forall a\in A, h\in H, f\in H^*.$$
Blattner and Montgomery showed in \cite{BlattnerMontgomery} that $A\# H \# H^* = B\# H^*   \simeq M_n(A),$ where $n=\mathrm{dim}(H)$ and $M_n(A)$ denotes the algebra of $n\times n$-matrices  over $A$. Hence if $H^*$ is strongly semisimple, then the  left $H^*$-module algebra $B=A\# H$ is semiprime provided $B\# H^*\simeq M_n(A)$ is semiprime, i.e. provided $A$ is semiprime.  Cohen-Fishman's question can be therefore rephrased as follows:
\begin{corollary} Every semisimple Hopf algebra is strongly semisimple if and only if  for any semisimple Hopf algebra $H$ and left $H$-module algebra $A$:
$$A \mbox{ is semiprime } \qquad \Longleftrightarrow \qquad A\# H \mbox{ is semiprime.}$$
\end{corollary}

Note that the subalgebra of invariants of $B=A\# H$ is $B^{H^*}=A$. Moreover $A\subseteq B=A\# H$ is what is called a $H^*$-Galois extension, which basically means that $A$ and $B\# H^*$ are Morita equivalent. While in general a $H^*$-Galois extension $A\subseteq B$ cannot be expressed as a smash product, but as a  crossed product, Cohen-Fishman's question could be generalized to $H$-Galois extension $B^{coH}\subseteq B$ with $H$ being semisimple. This more general approach has been considered in \cite{MontgomerySchneider} where in particular the Krull relations, i.e. maps between $\mathrm{Spec}(A)$ and $\mathrm{Spec}(B)$) were studied.

\subsection{Trivial Hopf algebras}
Note that if $H$ is a commutative $n$-dimensional semisimple Hopf algebra, then $H$ has  a complete orthogonal set of idempotents $\{p_1, \ldots, p_n\}$ of $H$ such that $Hp_i$ is one-dimensional. Let  $\{g_1, \ldots, g_n\}$ be the corresponding dual basis in $H^*$. By equation (\ref{dualHopfalgebraComultiplication}) one has that for all $1\leq i \leq n$:
$\Delta_{H^*}(g_i) = \sum_{s,t} g_i(p_sp_t) g_s \otimes g_t = g_i \otimes g_i$, i.e. the basis  $\{g_1, \ldots, g_n\}$ of $H^*$ consists of group-like elements. Hence  $H^*$ is a group algebra. 

Analogously one has for a cocommutative finite dimensional semisimple Hopf algebra $H$ that $H^*$ is commutative and semisimple by Larson and Radford's theorem and hence  $H\simeq {H^{**}}$  is a group  ring.

\begin{corollary}
Let $H$ be any semisimple Hopf algebra over $k$.
\begin{enumerate}
\item If $H$ is commutative, then there exists a group $G$ such that 
$H\simeq k[G]^*$.
\item If $H$ cocommutative, then there exists a group $G$ such that 
$H\simeq k[G]$.
\end{enumerate}
\end{corollary}

Trivial Hopf algebras are precisely those that are semisimple and commutative or cocommutative. For Cohen-Fishman's question it is therefore important to consider non- commutative, non-cocommutative  semisimple Hopf algebras.

\subsubsection{A non-trivial semisimple Hopf algebra and a non-trivial action}
The smallest example of a non-trivial semisimple Hopf  is the following (see \cite{KacPaljutkin} or \cite{Masuoka95})

\begin{example}[Kac-Paljutkin 1966, Masuoka 1995]
Let $G=C_2\times C_2$ be the Klein group with $C_2$ being the cyclic 
group of order $2$. Let $x$ and $y$ be a pair of generators of $G$. 
Denote by $R=\CC[G]$ the group algebra of $G$ over $\CC$ and let $\sigma$ 
be the involution of $R$ that swaps $x$ and $y$. The element $z^2 - 
\frac{1}{2}(1+x+y-xy)$ is central in the skew polynomial ring
$R[z;\sigma]$. Moreover the Hopf algebra structure of $R$ extends to a 
Hopf algebra structure of the quotient:
$$H_8=R[z; \sigma]/\langle z^2 - \frac{1}{2}(1+x+y-xy)\rangle,$$
where one sets
$$\Delta(z)=\frac{1}{2}(1\otimes 1 + 1\otimes x + y\otimes 1 - y\otimes 
x)(z\otimes z),\qquad S(z)=z, \qquad \epsilon(z)=1.$$
\end{example}
Note that $S=id$ shows that $H_8$ is semisimple by Larson and Radford's theorem.
Etingof and Walton's theorem \ref{EtingofWaltonTheorem} showed that any action of  semisimple Hopf algebra on a commutative domain is virtually given by a group action. If $A$ is a non-commutative domain, then there might exist an action of a semisimple Hopf algebra $H$  that does not act as a trivial Hopf algebra. The following example stems from Kirkman et al. \cite[{Example 7.4}]{kirkman09}:
\begin{example}[A non-trivial action of $H_8$ on the quantum plane]
Let $A=\CC_q[u,v]$ be the quantum plane at the parameter $q$ with $q^2=-1$. The eight-dimensional semisimple Hopf algebra $H_8$ acts on $A$ as follows:
\[\begin{array}{lcrp{10mm}lcrp{10mm}lcr} x\cdot u &=&-u,&&y\cdot u &=& 
u, &&z\cdot u &=& v\\
x\cdot v &=& v, &&y\cdot v &=& -v,&&z\cdot v &=& u.\end{array}\]
Note that
\begin{eqnarray*} z\cdot (uv) &=&
\frac{1}{2}((z\cdot u)(z\cdot v) + (z\cdot u)(xz\cdot v) + (yz\cdot 
u)(z\cdot v)-
(yz\cdot u)(xz\cdot v)) \\
&=& \frac{1}{2}(vu - vu - vu - vu) =-vu \neq vu =(z\cdot u)(z\cdot v)
\end{eqnarray*}
This shows that $z$ does not act as an algebra endomorphism.
\end{example}

\section{Conditions on $A$}
There are two possible ways to tackle Cohen-Fishman's question: by additional assumptions on the left $H$-module algebra $A$ (apart from being semiprime) and by additional conditions on the Hopf algebra $H$ (apart from being semisimple).  I will focus first on additional properties on the module algebra $A$.

\subsection{Separable extensions}

Maschke's theorem says that a group algebra $H=k[G]$ of a finite group $G$ is a semisimple Artinian ring if and only if  $\mathrm{char}(k) \nmid |G|$.  Sweedler proved an analogues theorem for Hopf algebras in \cite{Sweedler69}: 

\begin{theorem}[Sweedler]\label{SweelderTheorem}
The following statements are equivalent for a Hopf algebra over a field $k$:
\begin{enumerate}
\item[(a)] $H$ is a semisimple Artinian Hopf algebra;
\item[(b)] $H$ is a separable $k$-algebra;
\item[(c)] $k$ is a projective left $H$-module;
\item[(d)] $\exists t\in H: \forall h\in H: ht=\epsilon(h)t$ and $\epsilon(t)=1$.
\end{enumerate}
\end{theorem}
The element $t\in H$ such that $ht=\epsilon(h)t$ for all $h\in H$ is called a left integral in $H$.
Such integrals exists in a Hopf algebra if and only if the Hopf algebra is finite dimensional (see \cite{Pareigis, Lomp04, LarsonSweedler}). In the case of a group algebra $H=k[G]$, an integral corresponds to (a scalar multiple) of $t=\sum_{g\in G} g$. Note that a left integral $t$ is also a right integral in case $H$ is semisimple and hence $t$ is a central idempotent of $H$. The condition $(d)$ of Sweedler's theorem can be seen as a generalization of Maschke's theorem, because $\epsilon(t)=|G|1_H$ is non-zero if and only if  $\mathrm{char}(k)\nmid |G|$.

A generalization of a separable $k$-algebra is the notion of a separable ring extension that was introduced by Hirata and Sugano in \cite{HirataSugano}:

\begin{definition}[Hirata-Sugano, 1966]
A ring extension $R\subseteq S$ is \emph{separable} if the multiplication map $\mbox{mult}: S{\otimes_R} S \rightarrow S$ splits  as $S$-bimodule.
\end{definition}
Equivalently $R\subseteq S$ is a separable ring extension if there exists an element $\gamma\in {S{\otimes_R}S}$ that  is $S$-centralizing, i.e. $s\gamma =\gamma s$ for all $s\in S$, and that satisfies $\mbox{mult}(\gamma)=1$.
An element $\gamma$ in $S{\otimes_R} S$ with these properties is called a {\it separable idempotent}.

A characterization of semisimple Hopf algebras in terms of separable extensions is given as follows:
\begin{corollary}
$H$ is semisimple if and only if $A\subseteq A\# H$ is a separable extension for any left $H$-module algebra $A$.
\end{corollary}
\begin{proof}
Let $H$ be a semisimple Hopf algebra and let $t$ be an integral in $H$ with $\epsilon(t)=1$ from  theorem \ref{SweelderTheorem}(d). 
Let $\gamma=\sum_{(t)} {1\#S(t_1)}\: {\otimes_A} \: {1\# t_2} \in ({A\# H})\: {\otimes_A} \: ({A\# H}).$ 
If $\mu$  denotes the multiplication of $A\# H$, then
$$\mu(\gamma)=\sum_{(t)} (S(t_2)\cdot 1_A)\# S(t_1)t_3 
=1_A\# \sum_{(t)} S(t_1)\epsilon(S(t_2)) t_3  
=1_A\# \epsilon(t)1_H = 1_A \# 1_H.$$
Given an element $h\in H$ one has 
$h\otimes \Delta(t) =\sum_{(h,t)} h_1\otimes t_1h_2 \otimes t_2h_3$ (see for example \cite{Lomp04}). Thus 
$$h\sum_{(t)} S(t_1) \otimes t_2 =\sum_{(h,t)} h_1S(t_1h_2) \otimes t_2h_3 = \sum_{(t)} S(t_1) \otimes t_2h .$$
i.e. $\gamma(1\# h) =(1\# h)\gamma$. Also for any $a\in A$ one has
\begin{eqnarray*}\gamma a &=& \sum_{(t)} 1\# S(t_1) {\otimes_A} (1\#t_2)(a\#1)\\
&=& \sum_{(t)} (1\# S(t_1))(t_2\cdot a\# 1) {\otimes_A} 1\# t_3
=\sum_{(t)} ((t_2 S(t_3)\cdot a) \# S(t_1) {\otimes_A} 1\#t_4
 =a \gamma
\end{eqnarray*}
Hence $\gamma$ is a separable idempotent for $A\subseteq A\# H$.
\end{proof}

\subsection{Von Neumann regular algebras}
Separable extensions $R\subseteq S$ have the good property that any short exact sequence of left $S$-modules that splits as left $R$-modules also split as left $S$-modules (see \cite[Proposition 2.6]{HirataSugano}). Hence any left $S$-module that is projective as left $R$-module is also projective as left $S$-module by \cite[Proposition 1.6]{HirataSugano}. Moreover if $M$ is a left $S$-module that is flat as left $R$-module, then $M\simeq \lim P_\lambda$ for some finitely generated projective left $R$-modules $P_\lambda$. The modules $S{\otimes_R} P_\lambda$ are finitely generated projective left $S$-modules and hence 
$$S {\otimes_R} M \simeq  \lim (S {\otimes_R} P_\lambda)$$
is a flat left $S$-module. The map $\varphi: S\otimes_R M \rightarrow M$ with $s\otimes m = sm$ is left $S$-linear and splits as left $R$-module map with retraction map sending  $m$ to $1 {\otimes_R} m$ for all $m\in M$. Since $R\subseteq S$ is separable,  $\varphi$ also splits  as left $S$-module map, i.e. $M$ is isomorphic to a direct summand of a flat left $S$-module and therefore itself flat as left $S$-module.

This shows that if $R \subseteq S$ is separable and $R$ is either semisimple Artinian or von Neumann regular, then so is $S$.

\begin{corollary}\label{semisimple}
Let $H$ be a semisimple Hopf algebra acting on $A$. 
\begin{enumerate}
\item If $A$ is von Neumann regular, then $A\# H$  is von Neumann regular. 
\item If $A$ is semisimple Artinian, then $A\# H$  is semisimple Artinian.
\end{enumerate}
\end{corollary}
The second implication was first proved by M.Cohen and D.Fishman in \cite{CohenFishman} while the von Neumann regular case appeared in \cite{LompThesis}.

\subsection{Classical ring of quotient}
Corollary \ref{semisimple} showed that Cohen-Fishman's question has a positive answer for semiprime Artinian module algebras $A$. In this section I will recall a theorem by Skryabin and van Oystaeyen  from \cite{SkryabinVanOystaeyen}  that says that the same is true if $A$ is semiprime Noetherian. A natural question arises whether, in case $A$ embeds into an semisimple Artinian overring $Q$,  the Hopf action on $A$ extend to the overring  $Q$. Then  $Q\# H$ would be semiprime  as seen above and one intends to deduce that $A\# H$ is semiprime as well. Obviously the first choice for such overrings $Q$ are suitable rings of quotients of $A$.  Skryabin and van Oystaeyen showed in \cite{SkryabinVanOystaeyen} that Cohen-Fishman's question has a  positive answer if $A$ is a Noetherian semiprime left $H$-module algebra. Their result is based on the following theorem, that says that for an algebra $A$ with  an Artinian classical ring of quotient $Q$, any Hopf algebra action $H$ on $A$ can be extended to $Q$. Recall that the classical ring of quotients $Q$ of $A$, if exists, satisfies the following universal property: if $\varphi: A\rightarrow S$ is a ring homomorphism from $A$ to a ring $S$ such that for any non-zero divisor $a\in A$, the image $\varphi(a)$ is invertible in $S$, then there exists a unique ring homomorphism $\overline{\varphi}:Q\rightarrow S$ such that $\varphi_{|_A} = \overline{\varphi}_{|_A}$.

\begin{theorem}[Skryabin - Van Oystaeyen, 2006] If $A$ has a right Artinian classical ring of quotient $Q$, then any left 
Hopf module action on $A$ extends to $Q$.
\end{theorem}
\begin{proof}
I will shortly sketch the ideas of their proof. First of all an algebra $A$ is a left $H$-module algebra if and only if $A$ is a left $H$-module and $H$ ``measures" $A$, i.e. equations (\ref{modulealgebra1}) and (\ref{modulealgebra2}) are satisfied. The measuring of $H$ to $A$ can be also expressed through the existence of an algebra homomorphism
$$ \rho: A\rightarrow \mathrm{Hom}(H,A)$$ 
from $A$ to the convolution algebra $\mathrm{Hom}(H,A)$. The $H$-module action is then given by $h\cdot a := \rho(a)(h)$, for $a\in A, h\in H$. In order to obtain a measuring of $H$ on $Q$, Skryabin and Van Oystaeyen considered
the following diagram:
\[\xymatrix{
A \ar@{->}[rr]^{\rho} \ar@{^{(}->}[d] &&\mathrm{Hom}(H,A)\ar@{^{(}->}[d] \\
 Q \ar@{-->}[rr]^{\exists! \rho'} &&\mathrm{Hom}(H,Q) 
}\]
To prove the existence of the unique algebra map $\rho':Q\rightarrow \mathrm{Hom}(H,Q) $ one uses the universal property of $Q$ showing that $\rho(u)$ is invertible in $\mathrm{Hom}(H,Q)$ for any non-zero divisor $u\in A$. In order to do so, $H$ is replaced by a finite dimensional subcoalgebra $C$. This is possible because  a measuring only depends on the coalgebra structure of $H$ and furthermore because  $H$ is the union of its finite dimensional subcoalgebras $C$. Since $Q$ is Artinian, $\mathrm{Hom}(C,Q)$ is also Artinian and hence it is enough to show that $\rho(u)$ is not a zero divisor for any non-zero divisor $u\in A$. In a final step it is proven that the measuring given by $\rho'$ is actually a left $H$-module action on $Q$.
\end{proof}
Since a Noetherian semiprime algebra $A$ has a semisimple Artinian classical algebra of quotient $Q$, any action of a (semisimple) Hopf algebra $H$ extends to $Q$. By Corollary \ref{semisimple} $Q\# H$ is semisimple Artinian and also an Ore localization of $A\# H$. Thus $A\# H$ is semiprime.
\begin{corollary}[Skryabin - Van Oystaeyen, 2006]\label{GoldieTheorem}
If $A$ is semiprime right Noetherian and $H$ semisimple, then $A\# H$ is semiprime.
\end{corollary}

\begin{example}
Skryabin and Van Oystaeyen's result Corollary \ref{GoldieTheorem} applies in particular to the non-trivial action of $H_8$ on the quantum plane $A=\CC_q[u,v]$ for $q^2=-1$, showing that 
$\CC_q[u,v] \# H_8$ is semiprime.
\end{example}

\subsection{Gabriel localization}
It is natural to ask, when a Hopf algebra action extends to a quotient algebra of $A$ obtained by other methods than by Ore localization. For example when does the $H$-action extends to a localization $Q=A_{\cF}$ with respect to a Gabriel filter $\cF$? Recall that a filter $\cF$ of left ideals of a ring $A$ is called a Gabriel filter if the following two conditions are fulfilled:  
\begin{enumerate}
\item[(i)] $Ia^{-1}\in \cF$ for all $I\in \cF$ and $a\in A$;
\item[(ii)] $\forall {_AJ}\subseteq {_AA}: [ \exists I\in \cF, \forall a\in I: Ja^{-1}\in \cF] \Rightarrow J\in \cF$,
\end{enumerate}
where $Ja^{-1}=\{b\in A: ba\in J\}$. The algebra $A$ becomes a topological ring with respect to $\cF$, where the elements of $\cF$ are considered the neighborhoods of zero of the topology.

\begin{theorem}[Montgomery, 1993, Selvan 1994, Sidorov 1996]
Let $\cF$ be a Gabriel filter on a left $H$-module algebra $A$ over some Hopf algebra $H$.
If for each $h\in H$ the map 
$$\rho_h:A\longrightarrow A \qquad \mbox{ with } \qquad \rho_h(a)=h\cdot a, \forall a\in A$$
is continuous with respect to the topology induced by $\cF$, then the $H$-action extends to the localization $Q=A_{\cF}$.  In case $H$ is finite dimensional, the continuity of all maps $\rho_h$ is equivalent to the condition that any left ideal in $\cF$ contains an $H$-stable left ideal which still belongs to $\cF$.
\end{theorem}

For the proof see \cite{MontgomeryInvertible,Selvan, Sidorov}. Actually Sidorov proved in \cite{Sidorov} more generally that any $\cF$-continuous local action of a Hopf algebra on $A$ can be extended to a global action on $A_{\cF}$, where $H$ \emph{acts locally} on $A$ if the action of an element $h\in H$ is defined on some element of $\cF$.

In \cite{Rumynin93}, D. Rumynin claimed that an action of a semisimple Hopf algebra always extends to the maximal left ring of quotient of a left non-singular left $H$-module algebra $A$. Recall that the maximal left ring of quotient of a left non-singular ring is the Gabriel localization with respect to the filter $\cF$ of essential left ideals. In the case of $A$ being left non-singular, $A_\cF$ is the injective hull of $A$ as a left $A$-module. However there was a gap in the proof of \cite{Rumynin93} and an (unsuccessful) attempt to fix it in \cite{LompExtending} could only conclude that the Hopf action extends to the maximal left ring of quotients of $A$ if and only if $A_\cF$ coincides with the Gabriel localization of  $A$ with respect to the filter associated to the injective hull of $A$ as $A\# H$-module.

\subsection{Martindale ring of quotient}

The disadvantage of dealing with the classical or the maximal ring of quotient is that further assumptions like the Goldie or non-singularity condition have to be made. Martindale developed a theory of embedding any semiprime ring $R$ into a certain overring such that its  center $Z(R)$ is a von Neumann regular ring. M. Cohen, already in her early paper \cite{Cohen85}, extended the Hopf action of a Hopf algebra $H$ on a semiprime left $H$-module algebra $A$ to a certain subring of Martindale's quotient algebra.

\begin{theorem}[Cohen, 1985]
Let $A$ be any semiprime left $H$-module algebra. The $H$-action extends to
$$Q_0 =\lim \{ \mathrm{Hom}({_AI}, {_AA})\: \mid I \mbox{ is an $H$-stable ideal of $A$ with zero annihilator}\}.$$
\end{theorem}

J. Matczuk in \cite{Matczuk91} used the overring $Q_0$ to define the $H$-central closure of $A$ as the subalgebra generated by $A$ and $Z(Q_0)^H$. In \cite{LompCClosure} I adapted R. Wisbauer's approach from \cite{WisbauerPrime} to define the $H$-central closure of a semiprime $H$-module algebra on a certain self-injective module which had the advantage of using module theoretical results to tackle Cohen-Fishman's question. For this purpose one defines an algebra structure on $A\otimes A^{op} \otimes H$ as follows: for all $a,a',b,b'\in A$ and $h,g\in H$ set:
$$ (a\otimes b \otimes h)(a'\otimes b' \otimes g) =\sum_{(h)} a(h_1\cdot a') \otimes (h_3\cdot b')b \otimes h_2g.$$  
 It has been show in \cite{LompThesis} (see also \cite{ConnesMoscovici, Kadison, PanaiteVanOystaeyen}) that $A\otimes A^{op} \otimes H$ becomes an associative unital algebra with this multiplication, which I  denote by $A^e \bowtie H,$ where $A^e=A\otimes A^{op}$ stands for the enveloping algebra of $A$. Note that if $H$ is cocommutative, then $A^e$ becomes a left $H$-module algebra and $A^e\bowtie H = A^e \# H$ is the ordinary smash product. Moreover $A$ has a natural $A^e\bowtie H$-module structure such that its submodules are precisely its $H$-stable two-sided ideals. The following two theorems appeared  in \cite{LompThesis}  and were published in \cite{LompSmash}.
 
\begin{theorem}[Lomp, 2002]\label{selfinjective}
Let  $\widehat{A}$ be the self-injective hull of $A$ as left $A^e\bowtie H$-module. Then
\begin{enumerate}
\item $\widehat{A}$ is a left $H$-module algebra with subalgebra $A$;
\item $\widehat{A}$ is isomorphic to Matczuk's $H$-central closure, i.e. $$\widehat{A}\simeq \langle A, Z(Q_0)^H\rangle\subseteq Q_0;$$
\item $\mathrm{End}_{A^e\bowtie H}(\widehat{A}) \simeq Z(Q_0)^H$;
\item $Z(Q_0)^H$ is von Neumann regular and self-injective.
\end{enumerate}
\end{theorem}

As an application one can prove that Cohen-Fishman's question has a positive answer for commutative module algebras.
\begin{theorem}[Lomp, 2002]
Let $H$ be any semisimple Hopf algebra  and $A$ a commutative left $H$-module algebra. If $A$ is semiprime, then $A\# H$ is also semiprime.
\end{theorem} 
\begin{proof}
The idea of the proof as given in \cite{LompSmash} is as follows: $A$ being commutative implies that $\widehat{A}$  is commutative.
Zhu's result \cite{Zhu96} says that $\widehat{A}^H \subseteq \widehat{A}$ is an integral extension. Since $\widehat{A}^H$ is von Neumann regular by \ref{selfinjective}(4), also $\widehat{A}$ is von Neumann regular  (which in the commutative reduced case means that they are $0$-dimensional rings). As $\widehat{A}\subseteq \widehat{A}\# H$ is a separable extension, also $\widehat{A}\# H$ is von Neumann regular by \ref{semisimple}. As $\widehat{A}\# H$ is a central extension of $A\# H$, the later one is semiprime.
\end{proof}

Linchenko and Montgomery generalized the last theorem in \cite{LinchenkoMontgomery}  to semiprime $H$-module algebras that satisfy a polynomial identity.

\begin{theorem}[Linchenko-Montgomery, 2007]
Let $H$ be any semisimple Hopf algebra  and $A$ a left $H$-module algebra that satisfies a polynomial identity. If $A$ is semiprime, then $A\# H$ is also semiprime.
\end{theorem} 
\begin{proof}
Linchenko and Montgomery's original statement is more general than the statement that I give here (due to my general characteristic zero assumption). The basic ideas of their proof can be summarized as follows: The polynomial ring $A[t]$ is also a left $H$-module algebra, where $H$ acts trivially on $t$. Using Amitsur's theorem one concludes that $\mathrm{Jac}(A[t]) = 0$ and hence one can for simplicity assume $\mathrm{Jac}(A)=0$. Hence $\mathrm{Jac}(A\# H \# H^*)=0$ holds, using Blattner-Montgomery duality. Since $A\# H$ is a projective $A\# H \# H^*$-module (using the assumption of $H^*$ being semisimple), also $\mathrm{Rad}(A\# H)=0$ as $A\# H \# H^*$-module. This implies in particular that $\mathrm{Jac}(A\# H)$ does not contain any $H^*$-stable left ideal of $A\# H$. On the other hand, in  \cite{Linchenko} Linchenko proved that the Jacobson radical of a finite dimensional module algebra is stable under the Hopf action and the fact that primitive factors of PI-algebras are finite dimensional over their centre allows to show that $\mathrm{Jac}(A\# H)$ is $H^*$-stable and hence zero.
\end{proof} 

Linchenko and Montgomery's proof works also for finite dimensional involutive weak Hopf algebras (see  \cite{BorgesLomp}). 

\subsection{Large subalgebra of invariants}
A necessary condition for a positive answer to Cohen-Fishman's question is that strongly semisimple Hopf algebras $H$ have the property that semiprime left $H$-module algebras $A$ have a large subalgebra of invariants $A^H$ in the sense that any non-zero $H$-stable left ideal intersects $A^H$ non-trivially. To see this let $t$ be an integral of $H$ with $\epsilon(t)=1$. If $A$ is semiprime and $H$ strongly semisimple, then $A\# H$ is semiprime. Since $(I\# t)^2 \neq 0$ for any $H$-stable left ideal $I$ of $A$, one concludes that $0\neq (t\cdot I)  \subseteq I\cap A^H$.  It is not known whether this holds for any semiprime left $H$-module algebra $A$ over a semisimple Hopf algebra. For group actions this fact was first proved by Bergman and Isaacs in \cite{BergmanIsaacs}. For finite dimensional Hopf algebras with cocommutative coradical a similar statement had been proved by Beidar and Torrecillas in \cite{BeidarTorrecillas}. The following theorem, proved by Bahturin and Linchenko in \cite{BathurinLinchenko}, involves not necessarily unital module algebras.

\begin{theorem}[Bathurin-Linchenko, 1998]\label{Bathurin} The following statements are equivalent for a finite dimensional Hopf algebra $H$.
\begin{enumerate}
\item[(a)] Any, not necessary unital, left $H$-module algebra $A$ is nilpotent if $A^H$ is nilpotent.
\item[(b)] Any, not necessary unital, left $H$-module algebra $A$ satisfies a polynomial identity if $A^H$ does.
\item[(c)] $\mathrm{dim}(T/\langle T^H\rangle)<\infty$ for $T=T(H)/k$, where $T(H)$ is the tensor algebra of $H$.
\end{enumerate}
\end{theorem}

Bathurin and Linchenko also mentioned that any Hopf algebra that satisfies one of the conditions of theorem \ref{Bathurin} is semisimple, but were not able of proving the converse. However every Hopf algebra $H$ of theorem \ref{Bathurin} has the property that any semiprime left $H$-module algebra $A$ has a  large subalgebra of invariants$A^H$ because any $H$-stable left ideal $I$ of $A$ can be considered a non-unital left $H$-module algebra and hence if $I^H=I\cap A^H=\{0\}$ then $I$ would be nilpotent.

As a weak form of Cohen-Fishman's question one can formulate the following
 
\begin{question}
Let $H$ be a  semisimple Hopf algebra and $A$ a semiprime left $H$-module algebra. Does $A^H$ intersect any non-zero $H$-stable left ideal of $A$ non-trivially ?
\end{question}

\subsection{Prime and Simple smash products}
I conclude this section with short note on the characterization of prime and simple smash products as given by 
Osterburg et al. in \cite{OsterburgPassmanQuinn} in terms of the so-called \emph{Connes spectrum}. In the whole section we will assume that $k$ is an algebraically closed field of characteristic $0$.
Denote by $l(X)$ (resp. $r(X)$) the left (resp. right) annihilator of a subset $X$ in a ring $A$. A left $H$-module algebra $A$ is called $H$-prime if $l(I)=r(I)=\{0\}$ for all non-zero $H$-stable ideals $I$ of $A$, while $A$ is called $H$-simple if $\{0\}$ and $A$ are the only $H$-stable ideals of $A$. A \emph{hereditary subalgebra} of $A$ is a subalgebra $B$ of $A$ such that there exist an $H$-stable left ideal $L$ and an $H$-stable right ideal $R$ of $A$ such that $B=RL$. Let $\mathcal{H}(A,H)$ be the set of hereditary subalgebras $B$ of $A$ with $l(B)\cap B = r(B)\cap B = \{0\}$. From  \cite[Theorem 4.8 and 4.9]{OsterburgPassmanQuinn} one has 
\begin{theorem}[Osterburg-Passman-Quinn, 1992] Let $H$ be a  strongly semisimple Hopf algebra $H$, $A$ a left $H$-module algebra and set $e=1\#t$, where $t$ a left integral of $H$ with $\epsilon(t)=1$.
\begin{enumerate}
\item $A\# H$ is prime if and only if $A$ is $H$-prime and $BeB$ has zero left and right annihilator in $B\# H$ for all $B\in \mathcal{H}(A,H)$.
\item $A\# H$ is simple if and only if $A$ is $H$-simple and $BeB=B\# H$ for all $B\in \mathcal{H}(A,H)$.
\end{enumerate}
\end{theorem}
Osterburg et al. also expressed their result in terms of the so-called Connes spectrum whose definition I shall shortly present here. Since $H$ is semisimple and $k$ algebraically closed, $H$ decomposes into a direct product of full matrix rings over $k$. Let $\mathrm{Irr}(H)$ denote the set of irreducible representations $\pi: H \rightarrow M_{d_{\pi}}(k)$ of dimension $d_{\pi}$. For any $B\in \mathcal{H}(A,H)$ and $\pi \in \mathrm{Irr}(H)$ one defines
\begin{eqnarray*}
B^m_\pi &=& \left\{ X \in M_{d_\pi}(B) \mid \epsilon(h)X = \sum_{(h)} \pi(h_3)(h_1\cdot X)\pi(S^{-1}(h_2),\: \forall h\in H\right\}.\\
B^l_\pi &=& \left\{ X \in M_{d_\pi}(B) \mid \epsilon(h)X = \sum_{(h)} \pi(h_2)(h_1\cdot X), \: \forall h\in H\right\}.\\
B^r_\pi &=& \left\{ X \in M_{d_\pi}(B) \mid \epsilon(h)X = \sum_{(h)} (h_1\cdot X)\pi(S^{-1}(h_2),\: \forall h\in H\right\}.
\end{eqnarray*}
The \emph{Connes spectrum} is then defined as a subset of $\mathrm{Irr}(H)$ by
$$\mathrm{CS}(A,H) = \{\pi \in \mathrm{Irr}(H) \mid l(B^l_\pi B^r_\pi)\cap B^m_\pi =  r(B^l_\pi B^r_\pi)\cap B^m_\pi = \{0\}, \: \forall B\in \mathcal{H}(A,H)\}.$$
The \emph{strong Connes spectrum} is defined as 
$$\mathrm{CS}^*(A,H) = \{\pi \in \mathrm{Irr}(H) \mid B^l_\pi B^r_\pi = B^m_\pi, \: \forall B\in \mathcal{H}(A,H)\}.$$
The main result of \cite{OsterburgPassmanQuinn} is that in case $H$ is strongly semisimple, $A\#H $ is $H$-prime if and only if $A$ is $H$-prime and $\mathrm{Irr}(H)=\mathrm{CS}(A,H)$, while
$A\#H $ is $H$-simple if and only if $A$ is $H$-simple and $\mathrm{Irr}(H)=\mathrm{CS}^*(A,H)$.

\section{Conditions on $H$}
Instead of assuming further conditions on $A$ one might intend to classify semisimple Hopf algebras $H$ over $k$. Zhu proved in 1994 that  Hopf algebras of prime dimensions are group algebras (see \cite{Zhu94}) and Etingof and Gelaki proved in  \cite{EtingofGelaki} that  Hopf algebras whose dimension is a product of two prime numbers are trivial. 

\subsection{Semisolvable Hopf algebras}
Recall that a Hopf subalgebra $U$ of $H$ is called \emph{normal} if it is stable under the adjoint action, i.e. 
$$\forall h\in H: \mathrm{ad}_h(U) =\sum_{(h)} h_1US(h_2) \subseteq U.$$
If $H$ does not contain proper normal Hopf subalgebras, then it is called a \emph{simple Hopf algebra}. If $U$ is normal in $H$, then $\overline{H}=H/U^+$ becomes a Hopf algebra  with $U^+=U\cap \mathrm{Ker}(\epsilon)$.  Moreover $H$ can be recovered from $U$ and $\overline{H}$ as a \emph{crossed product}.
In \cite{MontgomeryWhitherspoon} Montgomery and Witherspoon defined a Hopf algebra $H$ to be  \emph{semisolvable} if it has a normal series
$$ k=H_0\trianglelefteq H_1 \trianglelefteq \cdots H_{m-1}\trianglelefteq H_m=H$$
with $H_{i-1}$ normal in $H_i$ such that $H_i/H_{i-1}^+$ is either commutative or cocommutative. In particular if $H$ is semisimple, then all these subquotients $H_i/H_{i-1}^+$ are trivial.
Masuoka showed in \cite{Masuoka} that every Hopf algebra of dimension $p^n$, with $p$ a prime number has a central group like element from which it follows that such Hopf algebras are semisolvable.

\begin{example}
Let $H_8=\CC[C_2\times C_2][z; \sigma]/\langle z^2 - \frac{1}{2}(1+x+y-xy)\rangle$ be the eight dimensional semisimple Hopf algebra from above. Then
$$U=\CC[C_2\times C_2] \trianglelefteq H_8 \qquad \mbox{ and } \qquad 
H_8/U^+ \simeq \CC[C_2].$$
is a normal series for $H_8$, i.e. $H_8$ is semisolvable.
\end{example}

The importance of the semisolvable condition for this paper is the following theorem due to Montgomery and Schneider  from \cite[Corollary 8.16]{MontgomerySchneider} which says that Cohen-Fishman's question has a positive answer for semisimple semisolvable Hopf algebras.

\begin{theorem}[Montgomery-Schneider, 1999]
Any semisimple semisolvable Hopf algebra is strongly semisimple.
\end{theorem}

Note that Montgomery and Schneider's notion of strongly semisimple Hopf algebras (as used in \cite{MontgomerySchneider}) is presumably stronger than the one used in this survey. They call a Hopf algebra $H$ strongly semisimple if $P\cap A^H$ is semiprime for any prime ideal $P$ of a left $H$-module algebra $A$. They showed in  \cite[Theorem 8.10]{MontgomerySchneider} that $A\# H$ is semiprime for any $H$-prime left $H$-module algebra $A$. Hence it is not difficult to see that Montgomery and Schneider's condition implies condition (\ref{implication}).

The first semisimple non-semisolvable Hopf algebra had been found by D. Nikshych in \cite{Nikshych} as a Drinfeld twist of the group algebra of $A_5$ (see the next subsection for the definition of twists). The group algebra $H=k[A_5]$ is a simple Hopf algebra and a suitable twist will keep simplicity while turning $H$ into a non-trivial simple, semisimple Hopf algebra. It is possible to find such twists for all simple groups. The smallest semisimple non-semisolvable Hopf algebra has dimension $36$ and was found by S. Natale and C. Galindo in \cite{GalindoNatale}. There exists a twist of  $H=k[S_3\times S_3]$ which turns $H$ into a non-trivial semisimple Hopf algebra that is simple as a Hopf algebra and hence not semisolvable. 

\subsection{Drinfeld twists}
In the last section I show that twists of strongly semisimple Hopf algebras are strongly semisimple.

\begin{definition}
A \emph{twist} for a Hopf algebra $H$ is an invertible element 
 $J\in H\otimes H$, such that
$$(J\otimes 1)(\Delta \otimes 1)(J)  =(1\otimes J)(1\otimes \Delta)(J),$$
$$(\epsilon\otimes 1)(J) =1 =(1 \otimes \epsilon)(J)$$
holds.  Given such a twist $J$, the comultiplication of $H$ can be deformed to obtain a Hopf algebra  $(H,m,\Delta^J,\epsilon, S^J)$ with
$$\Delta^J(h) :=J\Delta(h)J^{-1}, \qquad S^J(h):=US(h)U^{-1}$$
for all $h\in H$ with $U:=m(1\otimes S)(J)$.
\end{definition}

Since the algebra structure of a twisted Hopf algebra $H^J$ is unchanged, it is clear that $H^J$ is semisimple if and only if $H$ is semisimple. Given a twist $J$ of a Hopf algebra $H$ and a left $H$-module algebra $A$ its possible  to deform the multiplication of $A$ to obtain a left $H^J$-module algebra on $A$
\begin{definition} Let $(A,\mu,\eta)$  be a left $H$-module algebra and $J$ a twist for $H$. The new multiplication on $A$ 
defined by
$$a \cdot_J b :=\mu^J(a\otimes b) := \mu( J^{-1} \cdot (a\otimes b)) \:\:\mbox{ for all }a,b \in A.$$
makes $A$ a left $H^J$-module algebra, where  $J^{-1}\in H\otimes H$ acts componentwise on $A\otimes A$. The new $H^J$-module algebra is denoted by $A^J$.
\end{definition}
An elementary result first proven by S. Majid in \cite{Majid} shows that the smash products  $A\#H$ and $A^J \# H^J$ are isomorphic as algebras.

\begin{theorem}[Majid, 1997]
$A\# H \simeq A^J \# H^J$ as algebras.
\end{theorem}
Note that twisting is an invertible operation and one has $H={H^{J}}^{J^{-1}}$.
It is not difficult to prove that $A^J$ is semiprime if $A$ is. Thus if $H$ is strongly semiprime and $A$ is a semiprime left $H^J$-module algebra, then $A^{J^{-1}}$ is a semiprime left $H$-module algebra. Hence $A^{J^{-1}}\# H$ is semiprime and by Majid's isomorphism also $A\# H^J$ is semiprime.
This shows the following Corollary which appeared in \cite{LompSmash}.
\begin{corollary}
The class of strongly semisimple Hopf algebras is closed under Drinfeld twists.
\end{corollary}

Let $k$ be an algebraically closed field of characteristic zero.  Etingof and Gelaki showed in \cite{EtingofGelaki00} that any semisimple triangular Hopf algebra over $k$ is a Drinfeld twist of a group algebra. Thus these Hopf algebras are strongly semisimple. Moreover S. Natale showed in her monograph \cite{Natale} that any semisimple Hopf algebra $H$ of dimension less than $60$ over $k$ is semisolvable or a Drinfeld twist of a semisolvable Hopf algebra or the dual of a semisolvable one.

\section{Conclusion}
\begin{corollary}
Cohen-Fishman's question has a positive solution (over field of characteristic 0) for 
\begin{enumerate}
\item any semisimple Hopf algebra $H$ that is a twist of a semisolvable Hopf algebra;
\item any semiprime module algebra $A$ that either satisfies a polynomial identity or has an Artinian ring of quotients.
\end{enumerate}
\end{corollary}

In order to shed more light into Cohen-Fishman's question I suggest to consider the following questions/tasks:
\begin{enumerate} 
\item Do all semiprime module algebras over semisimple Hopf algebras have a large subalgebra of invariants ?
\item Find a semisimple Hopf algebra $H$ that is not a twist of a semisolvable Hopf algebra and a suitable $H$-action on  a semiprime algebra $A$. 

\item Extend Etingof-Walton's result of semisimple Hopf algebras on integral domains. What can be said about semisimple Hopf algebra actions on simple domains or free algebras in terms of their smash products?
\item Look at more general actions than Hopf algebra actions like actions of weak Hopf algebras, Hopfish algebras or  bialgebroids to find possible counterexamples. In particular does there exists a counterexample for Cohen-Fishman's question for the 13 dimensional weak Hopf algebra associated to the Lee-Yang fusion rule (see \cite[Section 5]{BohmSzlachanyi}) ?
\end{enumerate}

\end{document}